\documentclass[11pt]{article}
\usepackage{graphicx} 
\usepackage{amsmath,amssymb,amsfonts,enumerate,amsthm, amscd}
\usepackage[all]{xy}
\usepackage{graphicx, upgreek}
\usepackage{amsmath}
\usepackage{amsfonts}
\usepackage{amssymb}
\usepackage{hyperref}%
\usepackage{pst-node}
\usepackage{tikz-cd}
\newcommand{\subjclass}{%
	\textup{2020} \textbf{Mathematics Subject Classification:}}
\title{On $\widetilde{\operatorname{Spec}(M)}$ Topology of Module $M$ over Commutative Rings}
\author{Dilara Erdemir, Suat Koç, Ünsal Tekir and Mesut Buğday}
\date{}

\begin{document}

\maketitle
\begin{center}
		{\footnotesize 
			
			 Department of Mathematics, Yıldız Technical University, Istanbul, Turkey \\
			
			E-mail: dilaraer@yildiz.edu.tr \\

             Department of Mathematics, Marmara University, Istanbul, Turkey \\
E-mails: suat.koc@marmara.edu.tr \\ utekir@marmara.edu.tr \\ mesut.bugday@marmara.edu.tr
		}
	\end{center}
\bigskip
	\noindent
	{\small{\bf ABSTRACT.}
		Let $R$ be a commutative ring with unity and $M$ be an $R$-module. In this study, we construct the $\widetilde{\operatorname{Spec}(M)}$ topology using the prime spectrum of module $M$ and multiplicatively closed subsets of $R$ with the closed sets $\widetilde{V}(S)=\{P \in \operatorname{Spec}(M)|\ (P:M) \cap S_i \neq \emptyset\text{ for all} \ i \in I \}$ with the open sets $\widetilde{D}(S_i):=\{P \in \operatorname{Spec}(M)|\ (P:M) \cap S_i = \emptyset \}$ where $S=\{S_i\}_{i \in I}$ is a family of multiplicatively closed subsets of $R$. We investigate connections between the algebraic properties of $R$-module $M$ and the topological properties of $\widetilde{\operatorname{Spec}(M)}$. We examine specifically the separation axioms, connectivity, nested and Lindelöf property together with quasi-compactness as well as the isolated, closure, interior and limit points of $\widetilde{\operatorname{Spec}(M)}$. Moreover, in the last section, we provide an example of a Lindelöf space which is not quasi-compact by means of $\widetilde{\operatorname{Spec}(M)}$.       
	}
	
	\medskip
	\noindent
	{\small{\bf Keywords}{:} 
	Prime spectrum, multiplicatively closed subset, prime avoidance, connected space, discrete space

	}
\section{Introduction}
Construct topologies on algebraic structures and investigating the fundamental and topological properties of these topologies is a rather broad studied topic in commutative algebra and algebraic topology. The connections between topology types and algebraic structures are of interest to mathematicians. It is also possible to use topology to investigate and solve some algebraic problems, and algebraic topology also allows building a bridge between geometry and algebra. Another importance of algebraic topology is that it analyzes complex topological problems through algebraic structures.

 The set of all prime ideals of a ring is called the prime spectrum of the ring. The Zariski topology is a topology defined on the prime spectrum of commutative rings and was first introduced by Oscar Zariski in 1952. The connections between the properties of algebraic structures and the Zariski topology were studied extensively in many studies, see \cite{Clzar,LuYu, McMo2, Uzar, SzarS}. The Zariski topology on the prime spectrum of $R$-module $M$ was first initiated by Lu \cite{Lu} in which for any submodule $N$, the closed sets of the topology are the varieties $V(N)=\{P \in \operatorname{Spec}(M) |\ N \subseteq P\}$. In addition to extending the Zariski topology to modules, researchers also examined the connections of Zariski topologies on modules with different types of submodules through studies such as classical S-Zariski topology and S-Zariski topology on S-spectrum of modules and so on. 
 
 Recently, Parsa and Moghimi \cite{ParsaTopb} introduced $\widetilde{\operatorname{Spec(R)}}$ topology that is finer than the Zariski topology and constructed by using the prime spectrum topology and multiplicative closed subsets of a commutative ring $R$.

In this study, we focus on modules over commutative rings with unity. Let $M$ be a module over a commutative ring $R$. We say a proper submodule $P$ of $M$ is a prime submodule of $M$ if $rm\in P$ for some $r\in R$ and $m\in M$ then either $m\in P$ or $r\in (P:M)$ where $(P:M)=\{r\in R: rM\subseteq P\}$. Moreover, if $P$ is a prime submodule then $(P:M)$ is a prime ideal, see \cite{Lu3}. We denote the set of all prime submodules of a module $M$ by $\operatorname{Spec}(M)$. A subset $S$ of $R$ is said to be a multiplicatively closed subset of $R$ whenever $1_R \in S$ and when $x,y \in S$, then $xy \in S$ \cite{Sharp}. We first define the following sets $\widetilde{V}(S):=\{P \in \operatorname{Spec}(M)|\ S_i \ \cap (P:M) \neq \emptyset \ \text{for all} \ S_i \in S\}$ and $\widetilde{D}(T):=\{P \in \operatorname{Spec}(M)|\ (P:M) \cap T = \emptyset \}$ and then prove that $\widetilde{V}(S)$ are closed subsets of $\widetilde{\operatorname{Spec}(M)}$ in the next theorem and also show that the sets $\widetilde{D}(S)$ are elements of the basis of this topology in Theorem \ref{opset} where $S=\{S_i\}_{i \in I}$ is a family of multiplicatively closed subsets of $R$ and $T$ is a multiplicatively closed subset of $R$. Moreover, this topology is called $\widetilde{\operatorname{Spec}(M)}$ topology of $R$-module $M$. One can see that this construction extends naturally to the $\widetilde{\operatorname{Spec(R)}}$ topology into modules over commutative rings by means of prime spectrum of $M$ and multiplicatively closed subsets of $R$.

The main goal of this paper is to find algebraic and topological connections between module $M$ and $\widetilde{\operatorname{Spec}(M)}$. For example, we show that isolated points of $\widetilde{\operatorname{Spec}(M)}$ are minimal prime submodules of $M$ (Lemma \ref{isolated}). In particular, we show for multiplication modules that $M$ is a zero-dimensional module if and only if each $P\in\widetilde{\operatorname{Spec}(M)}$ is an isolated point (Corollary \ref{isolated2}). Next, we discuss nested property, connectivity and the closure, interior and limit points of $\widetilde{\operatorname{Spec}(M)}$ (See Theorem \ref{nested}, Theorem \ref{connected}, Theorem \ref{intcl}).

In the last section, we study each separation axiom together with quasi-compactness and Lindelöfness (See Theorem \ref{T0},\ref{T1},\ref{Discrete},\ref{T3},\ref{tcomp},\ref{lindelöf},\ref{tLin} and Corollary \ref{comp}). Lastly, we present a method to produce an example of a Lindelöf space which is not quasi-compact.
\section{Fundamental properties of $\widetilde{\operatorname{Spec}(M)}$  topology}
Throughout this work, $M$ is always a module over a commutative ring $R$ with unity. We first fix the notations as follows. Let $I$ be an arbitrary index set and $S$ be a family of multiplicative subsets of $R$ so that $S=\{S_i\}_{i\in I}$. Next, we define the following set $\widetilde{V}(S):=\{P \in \operatorname{Spec}(M)|\ S_i \cap (P:M) \neq \emptyset \text{ for all } S_i\in S\}$ and construct $\widetilde{\operatorname{Spec}(M)}$ topology in terms of closed sets $\widetilde{V}(S)$ in the following theorem.

\theoremstyle{definition}
\newtheorem{theorem}{\textbf{Theorem}}
\theoremstyle{definition}
\newtheorem{lemma}{\textbf{Lemma}}
\newtheorem{corollary}{\textbf{Corollary}}
\begin{theorem}\label{closedsets}
Let $M$ be an $R$-module. Then the following statements hold.
\begin{enumerate}
    \item[(i)] $\widetilde{V}(\{1\})=\emptyset$.
    \item[(ii)] $\widetilde{V}(S)=\operatorname{Spec}(M)$ if $I=\emptyset$.
    \item[(iii)] $\bigcap_{j \in J}\widetilde{V}(S_j)=\widetilde{V}(\bigcup_{j \in J}S_j)$, where $S_j$ is a family of multiplicative closed subset of $R$ for every $j\in J$.
    \item[(iv)] If $S$ and $S^{'}$ are families of multiplicatively closed subsets of $R$, then $\widetilde{V}(S) \cup \widetilde{V}(S^{'})=\widetilde{V}(\{R-(P:M)|\ (P:M) \cap A = \emptyset \ \text{and } (P:M) \cap A^{'} = \emptyset \ \text{for some} \ A \in S\text{ and } A^{'} \in S^{'}\})$.
\end{enumerate}
\end{theorem}
\begin{proof}
The first two statements are straightforward.
\begin{enumerate}
        \item[(iii)] Let $P \in \widetilde{V}(\bigcup_{j \in J}S_j)$ then $S \cap (P:M) \neq \emptyset$ for all $S \in \bigcup_{j \in J} S_j$. It follows that $S_{j,i} \cap (P:M) \neq  \emptyset$ for each $S_{j,i}\in S_j$. This implies that $P \in \widetilde{V}(S_j)$ and so $P \in \bigcap_{j \in J}\widetilde{V}(S_j)$ for all $j \in J$. On the other hand, if $P \in \bigcap_{j \in J}\widetilde{V}(S_j)$ then $S_{j,i} \cap (P:M) \neq \emptyset$ for all $S_{j,i} \in S_j$. Hence we obtain that
        $P \in \widetilde{V}(\bigcup_{j \in J}S_j)$ and so $\bigcap_{j \in J}\widetilde{V}(S_j) \subseteq \widetilde{V}(\bigcup_{j \in J}S_j)$. 
        \item[(iv)] Let $T:=\{R-(P_i:M)|P_i \in \operatorname{Spec}(M) \ \text{such that} \ (P_i:M) \cap A= \emptyset \ \text{and} \ (P_i:M) \cap A^{'}= \emptyset \ \text{for some} \ A \in S \ \text{and} \ A^{'} \in S^{'} \}$, and take an element $N \in \widetilde{V}(T)$. Now, we will show that $N\in\widetilde{V}(S)\cup\widetilde{V}(S')$. Suppose that $N\notin\widetilde{V}(S)\cup\widetilde{V}(S')$. Then we have $N\notin\widetilde{V}(S)$ and $N\notin\widetilde{V}(S')$. Then there exists $A\in S$ and $A'\in S'$ such that $(N:M)\cap A=\emptyset$ and $(N:M)\cap A'=\emptyset$. Now, put $T_i=R-(N:M)$. By the definition of $T$, note that $R-(N:M)=T_i\in T$ and also $(N:M)\notin\widetilde{V}(T_i)$ which implies that $N\notin\widetilde{V}(T)$. This is a contradiction. Thus, we conclude that $\widetilde{V}(T)\subseteq\widetilde{V}(S)\cup\widetilde{V}(S')$. For the other containment, choose $N\in\widetilde{V}(S)\cup\widetilde{V}(S')$. Without loss of generality, we may assume that $N\in\widetilde{V}(S)$. Then we have $(N:M)\cap A\neq\emptyset$ for all $A\in S$. Choose a prime submodule $P$ of $M$ such that $(P:M)\cap A=\emptyset$ and $(P:M)\cap A'=\emptyset$ for some $A\in S$ and $A'\in S'$. Now, we will show that $N\in\widetilde{V}(R-(P:M))$, that is, $(N:M)\cap (R-(P:M))\neq\emptyset$. Suppose that $(N:M)\subseteq (P:M)$. Since $(P:M)\cap A=\emptyset$, we have $(N:M)\cap A=\emptyset$ which implies that $N\notin\widetilde{V}(S)$. This is a contradiction. Thus, we have $N\in\widetilde{V}(R-(P:M))$ for all $R-(P:M)\in T$ which implies that $N\in\widetilde{V}(T)$, that is, $\widetilde{V}(S)\cup\widetilde{V}(S')\subseteq\widetilde{V}(T)$. Hence, we conclude that $\widetilde{V}(S)\cup\widetilde{V}(S')=\widetilde{V}(T)$. 
\end{enumerate}    
\end{proof}
The preceding theorem indicates that $\widetilde{V}(S)$ where $S$ is a family of multiplicatively closed subsets of $R$ are closed under finite union and arbitrary intersection. For this reason, $\widetilde{V}(S)$ can be considered as closed sets that satisfy the axioms of a topology on $\operatorname{Spec}(M)$, which is called $\widetilde{\operatorname{Spec}(M)}$ topology of $R$-module $M$. We remark that $\widetilde{\operatorname{Spec}(M)}$ corresponds to the $\widetilde{\operatorname{Spec}(R)}$ topology in case of $R$-module $R=M$. Also, we note here that we choose module $M$ such that $\operatorname{Spec}(M)\neq \emptyset$ otherwise $\widetilde{\operatorname{Spec}(M)}$ will be an empty space.
\begin{theorem}\label{opset}
Let $M$ be an $R$-module. Then, $$B=\{\widetilde{D}(S)|\ S \ \text{is a multiplicatively closed subset of} \ R\}\text{ is a basis for $\widetilde{\operatorname{Spec}(M)}$.}$$   
\end{theorem}
\begin{proof}
It is clear that $\widetilde{D}(S)$ is open in $\widetilde{\operatorname{Spec}(M)}$ for any multiplicatively closed subset $S$ of $R$ from Theorem \ref{closedsets}. Moreover, $$\operatorname{Spec}(M)-\widetilde{V}(S^{'})=\operatorname{Spec}(M)-\bigcap_{A \in S^{'}}\widetilde{V}(A)$$ \\ $$=\bigcup_{A \in S^{'}}(\operatorname{Spec}(M)-\widetilde{V}(A))=\bigcup_{A \in S^{'}}\widetilde{D}(A)$$ for any family $S^{'}$ of multiplicatively closed subsets of $R$. Hence, $B$ is a basis for $\widetilde{\operatorname{Spec}(M)}$. 
\end{proof}
Recall from \cite{Kap} that a multiplicatively closed subset $S$ of $R$ is said to be a \textit{saturated} if $xy \in S$ for some $x,y\in R$ implies that $x,y \in S$. It is clear that the intersection of any collection of saturated multiplicatively closed subsets is also a saturated multiplicatively closed set, and the intersection of all saturated multiplicatively closed subsets containing a multiplicatively closed subset $S$ of $R$, denoted by $\hat{S}$, is said to be a saturation of $S$. It is well known that $\hat{S}=\{r \in R : rs \in S \ \text{for some} \ s \in S\}$, and $S$ is a saturated multiplicatively closed subset of $R$ if and only if $S= \hat{S}$.
\theoremstyle{definition}
\newtheorem{proposition}{\textbf{Proposition}}
\begin{proposition}\label{basis2}
Let $M$ be an $R$-module. Then the following statements hold.
\begin{enumerate}
\item[i)] $\mathcal{B}=\{\widetilde{D}(S)|\ S \ \text{is a saturated multiplicatively closed subset of} \ R\}$ is a basis for the topology $\widetilde{\operatorname{Spec}(M)}$.
\item [ii)] $\mathcal{B}'=\{\widetilde{D}(R-(P:M)): P\in \operatorname{Spec}(M)\}$ is a basis for the topology $\widetilde{\operatorname{Spec}(M)}$.
\end{enumerate}
\end{proposition}
\begin{proof}
\begin{enumerate}
\item[i)] Let $S$ be a saturated multiplicatively closed set and consider $\widetilde{D}(S)\in \widetilde{\operatorname{Spec}(M)}$. Then by considering previous theorem, we have $\widetilde{D}(S)\in B$ since $S$ is a multiplicatively closed set. Let $S$ be a multiplicatively closed set and consider $\widetilde{D}(S)\in B$. It is easy to see that $\widetilde{D}(S)=\widetilde{D}(\hat{S})$ and $\hat{S}$ is always a saturated multiplicatively closed set. This gives $\widetilde{D}(S)=\widetilde{D}(\hat{S})\in\mathcal{B}$ which implies that $B=\mathcal{B}$, and the claim follows.
\item[ii)] It is clear that $\mathcal{B}'\subseteq\mathcal{B}$ since $R-(P:M)$ is saturated multiplicatively closed set for every $P\in \operatorname{Spec}(M)$, and thus $\mathcal{B}'$ consists of open sets in $\widetilde{\operatorname{Spec}(M)}$. Let $O$ be an open set in $\widetilde{\operatorname{Spec}(M)}$. Then by Theorem \ref{opset}, we have $O=\bigcup_{S_{i} \in S}\widetilde{D}(S_{i})$ for some family of multiplicatively closed set $S$ of $R$. Let $P\in O$ then there exists $S_{i}\in S$ such that $P\in\widetilde{D}(S_{i})$, that is, $(P:M)\cap S_{i}=\emptyset$. Now, take $Q\in\widetilde{D}(R-(P:M))$. This gives $(Q:M)\subseteq (P:M)$ and so $Q\in\widetilde{D}(S_{i})$. Thus we conclude that $P\in\widetilde{D}(R-(P:M))\subseteq\widetilde{D}(S_{i})\subseteq O$. Hence, $\mathcal{B}'$ is a basis for $\widetilde{\operatorname{Spec}(M)}$.
\end{enumerate}
\end{proof}
\theoremstyle{definition}
\newtheorem{example}{\textbf{Example}}
\begin{example}\label{examp}
    	Consider the $R=\mathbb{Z}_6$-module $M=\mathbb{Z}_6$. One may observe that $X=\operatorname{Spec}(\mathbb{Z}_6)=\{(\overline{2}),(\overline{3})\}$ together with $((\overline{2}):\mathbb{Z}_6)=(\overline{2})$ and $((\overline{3}):\mathbb{Z}_6)=(\overline{3})$. This implies that $\widetilde{D}(\mathbb Z_6-(\bar 2))=\{(\bar 2)\}$ and $\widetilde{D}(\mathbb Z_6-(\bar 3))=\{(\bar 3)\}$. As a result,  $\widetilde{\operatorname{Spec}(\mathbb Z_6)}=\{\emptyset,\{(\overline{2})\},\{(\overline{3})\}, \operatorname{Spec}(\mathbb{Z}_6)\}$ is a discrete topology by Proposition \ref{basis2} ii).
\end{example}
Recall that a topological space $(X,\tau)$ is an Alexandrov space if $\bigcap_{O \in \tau} O \in \tau$ for any $O \subseteq \tau$, or equivalently, every point $x\in X$ has a smallest neighborhood \cite{Arena}.
\begin{lemma}\label{neigh}
 Let $M$ be an $R$-module. Then, $\widetilde{\operatorname{Spec}(M)}$ is an Alexandrov space. Moreover, $\widetilde{D}(R-(P:M))$ is the smallest neighborhood for each $P \in \widetilde{\operatorname{Spec}(M)}$ .   
\end{lemma}
\begin{proof}
Let $\{U_i\}_{i \in I}$ be a family of open sets in $\widetilde{\operatorname{Spec}(M)}$ and $P \in \bigcap_{i \in I}U_i$. Then, there is a $\widetilde{D}(S_i) \in B$ such that $P \in \widetilde{D}(S_i) \subseteq U_i$ for all $i \in I $. This implies that $S_i \subseteq R-(P:M)$ and so $\widetilde{D}(R-(P:M)) \subseteq \widetilde{D}(S_i)$ for all $i \in I$. Hence, the desired result is achieved.
\end{proof}
A point $x$ of a topological space $X$ is called an \textit{isolated point} if $\{x\}$ is an open set in $X$ \cite{munkres}. Recall that a prime submodule $P$ of $M$ is a minimal prime submodule if $P$ is the minimal element of $\operatorname{Spec}(M)$ with respect to inclusion \cite{Sharp}. 
\begin{lemma}\label{isolated}
Isolated points of $\widetilde{\operatorname{Spec}(M)}$ are minimal prime submodules. In particular, if $M$ is a multiplication module, then every minimal prime submodule is an isolated point of $\widetilde{\operatorname{Spec}(M)}$.
\end{lemma}
\begin{proof}
Let $P\in \operatorname{Spec}(M)$ be an isolated point. Then $\widetilde{D}(R-(P:M))=\{P\}$ by Lemma \ref{neigh}. Moreover, if there is a prime submodule $Q$ of $M$ such that $Q\subseteq P$, then $(Q:M)\subseteq(P:M)$ which implies that $Q\in \widetilde{D}(R-(P:M))=\{P\}$. Thus we have $P=Q$, that is, $P$ is a minimal prime submodule of $M$. Now, suppose that $M$ is a multiplication module and $P$ is a minimal prime submodule of $M$. Let $Q\in\widetilde{D}(R-(P:M))$. Then we have $(Q:M)\subseteq (P:M)$. Since $M$ is multiplication, we conclude that $Q=(Q:M)M\subseteq P=(P:M)M$. As $P$ is a minimal prime submodule, we get $P=Q$ which implies that $\widetilde{D}(R-(P:M))=\{P\}$ is an open set. Hence, $P$ is an isolated point. 
\end{proof}
An $R$-module $M$ is said to be zero dimensional if every prime submodule is maximal, or equivalently, every prime submodule is minimal. Now, we characterize zero dimensional multiplication modules in terms of isolated points of $\widetilde{\operatorname{Spec}(M)}$.

\begin{corollary}\label{isolated2}
Let $M$ be a multiplication module. Then $M$ is a zero dimensional module if and only if every point of $\widetilde{\operatorname{Spec}(M)}$ is an isolated point, that is, $\widetilde{\operatorname{Spec}(M)}$ is a discrete space.      
\end{corollary}

\begin{proof}
Suppose that $P\in \operatorname{Spec}(M)$ and $M$ is a zero dimensional module. Then, $P$ is an isolated point by the previous lemma. For the converse, assume that $P\in \operatorname{Spec}(M)$ is an isolated point of $\widetilde{\operatorname{Spec}(M)}$. Then, $P$ is a minimal prime submodule of $M$ again by the previous lemma. Thus, $M$ is a zero dimensional module. 
\end{proof}

Recall that a topological space $(X,\tau)$ is called a \textit{nested space} if every open set of $X$ is comparable with respect to inclusion \cite{Ricmond}. Moreover, recall from \cite[Lemma 1]{YiKo} that $X$ is called nested if and only if basis elements of $(X,\tau)$ are comparable.
\begin{theorem}\label{nested}
$\widetilde{\operatorname{Spec}(M)}$ is a nested space if and only if every element of the set $\{(P:M)\ |\ P\in \operatorname{\operatorname{Spec}(M)} \}$ is totally ordered by inclusion.
\end{theorem}
\begin{proof}
$(\Leftarrow):$ One may have $(P:M)\subseteq (Q:M)$ or $(Q:M)\subseteq (P:M)$ for some $P,Q \in \operatorname{\operatorname{Spec}(M)}$. This implies that $\widetilde{D}(R-(P:M))\subseteq \widetilde{D}(R-(Q:M))$ or $\widetilde{D}(R-(Q:M))\subseteq \widetilde{D}(R-(P:M))$.
Hence, $\widetilde{\operatorname{Spec}(M)}$ is nested by Proposition \ref{basis2} ii) and \cite[Lemma 1]{YiKo}.

$(\Rightarrow):$ Let $(P:M)\not \subseteq(Q:M)$ and $(Q:M)\not \subseteq(P:M)$ for some $P,Q \in \operatorname{Spec}(M)$. It follows that $P\not\in \widetilde{D}(R-(Q:M))$ and $Q\not\in \widetilde{D}(R-(P:M))$. However, since $\widetilde{\operatorname{Spec}(M)}$ is nested, one may have $P\in \widetilde{D}(R-(P:M))\subseteq \widetilde{D}(R-(Q:M))$ which leads to a contradiction.
\end{proof}
\begin{lemma}\label{contmap}
Let $M$ and $M^{'}$ be $R$-modules, and $f: M \rightarrow M^{'}$ be a module epimorphism. Then, $f$ induces a continuous map $\widetilde{\theta}:\widetilde{\operatorname{\operatorname{Spec}(M}')} \rightarrow \widetilde{\operatorname{Spec}(M)}$ which is defined by $\widetilde{\theta}(P')=f^{-1}(P')$ for every prime submodule $P'$ of $M'$.
\end{lemma}
\begin{proof}
Consider the following map $\widetilde{\theta}:\widetilde{\operatorname{\operatorname{Spec}(M}')} \rightarrow \widetilde{\operatorname{Spec}(M)}, \ \widetilde{\theta}(P^{'})=f^{-1}(P^{'})$ for any prime submodule $P^{'}$ of $M^{'}$ and $S$ be a multiplicatively closed subset of $R$. Now, we will show that $\widetilde{\theta}^{-1}(\widetilde{D}_{M}(S))=\widetilde{D}_{M^{'}}(S)$. Let $P^{'} \in \widetilde{\theta}^{-1}(\widetilde{D}_M(S))$. Then $\widetilde{\theta}(P^{'})=f^{-1}(P^{'})\in\widetilde{D}_{M}(S)$. This implies that $(f^{-1}(P^{'}):M) \cap S= \emptyset$. Since $f$ is surjective, note that $(f^{-1}(P^{'}):M)=(P^{'}:M)$. Thus, we conclude that $(P^{'}:M) \cap S= \emptyset$, that is, $P^{'} \in \widetilde{D}_{M^{'}(S)}$. Then we have $\widetilde{\theta}^{-1}(\widetilde{D}_{M}(S))\subseteq\widetilde{D}_{M^{'}}(S)$. In a similar manner, one can obtain the reverse inclusion. Thus, we have the equality $\widetilde{\theta}^{-1}(\widetilde{D}_{M}(S))=\widetilde{D}_{M^{'}}(S)$. Hence, the inverse image of any open set in $\widetilde{\operatorname{\operatorname{Spec}(M}')}$ is again open in $\widetilde{\operatorname{\operatorname{Spec}(M)}}$.
\end{proof}

\begin{theorem}\label{homeo}
Let $M$ be a finitely generated injective module over a Noetherian ring $R$ and $S=\{f^n| \ n \geq 0\}$ a multiplicatively closed subset of $R$. If $\widetilde{D}(S)$ is considered as a subspace of $\widetilde{\operatorname{Spec}(M)}$, then $\widetilde{D}(S)$ is homeomorphic to $\widetilde{\operatorname{Spec}(\text{$S^{-1}M$}})$.  
\end{theorem}
\begin{proof}
Let $\pi:M \rightarrow S^{-1}M$ and $\varphi:R \rightarrow S^{-1}R$ be natural module and ring homomorphisms, respectively. It is clear that $\pi$ is surjective by \cite[Lemma 3.3]{Rob}. Then, $\pi$ induces the bijective continuous map $\widetilde{\theta}:\widetilde{Spec}(S^{-1}M)\rightarrow\widetilde{D}(S), S^{-1}P \mapsto \pi^{-1}(S^{-1}P) $ by Lemma \ref{contmap} and \cite[Proposition 2.2-(i) and Proposition 2.17]{Sprime}. Now, it is enough to show that $\widetilde{\theta}$ is an open map. Let $T$ be a multiplicatively closed subset of $S^{-1}R$. Then,
{\small
\begin{align*}
  \widetilde{\theta}(\widetilde{D}_{S^{-1}M}(T))=\{\pi^{-1}(S^{-1}P)|S^{-1}P \in \widetilde{D}_{S^{-1}M}(T)\} \\
  =\{\pi^{-1}(S^{-1}P)|(\pi^{-1}(S^{-1}P):M) \cap S= \emptyset \ \text{and} \ (S^{-1}P:S^{-1}M)
   \cap T= \emptyset\} \\
   =\{\pi^{-1}(S^{-1}P)|(\pi^{-1}(S^{-1}P):M) \cap S= \emptyset \ \text{and} \ (\pi^{-1}(S^{-1}P):M)
   \cap \varphi^{-1}(T)= \emptyset\} \\
  = \widetilde{D}_{M}(S) \cap \widetilde{D}_{M}(\varphi^{-1}(T))
\end{align*}\par}
which is open in $\widetilde{D}(S)$.
\end{proof}
Let $R$ be a ring and $\Lambda \subseteq \operatorname{Spec}(R)$. We say that $\Lambda$ satisfies prime avoidance (PA for short) if $I \subseteq \bigcup_{p \in \Lambda}p$ for some ideal $I$ of $R$, then $I \subseteq p$ for some $p \in \Lambda$\cite{justin}.

\begin{theorem}\label{priavo}
 Let $M$ be an $R$-module, $\{P_i\}_{i \in I}, \ \{Q_j\}_{j \in J}$ be two collections of prime submodules of $M$ and $S=R-\bigcup_{i \in I}(P_i:M), \ T=R-\bigcup_{j \in J}(Q_j:M)$. Then
 \begin{enumerate}
     \item[(i)] If $\{(P_i:M)\}_{i \in I} \cup \{(Q_j:M)\}_{j \in J}$ satisfies $PA$, then $\widetilde{D}(S) \cup \widetilde{D}(T)=\widetilde{D}(S \cap T)$.
     \item[(ii)] Suppose that $\{(P_i:M)\}_{i \in I}$ is a family of incomparable primes satisfying $PA$ and $M$ is a finitely generated faithful module. Then $\widetilde{D}(S)=\{P \in \operatorname{Spec}(M)|(P:M)=(P_{i_0}:M) \ \text{for some} \  i_0 \in I\}$ if and only if $\{(P_i:M)\}_{i \in I} \subseteq \operatorname{Min}(R)$.
     \item[(iii)] If $\operatorname{Max}(R)=\{(P_i:M)\}_{i \in I}$, then $\widetilde{D}(S)=\widetilde{\operatorname{Spec}(M)}$. 
 \end{enumerate}
\end{theorem}
\begin{proof}
  \begin{enumerate}
      \item[(i)] First, note that
      \begin{align*}
         S \cap T=[R-\bigcup_{i \in I}(P_i:M)] \cap [R-\bigcup_{j \in J}(Q_j:M)] \\
      =R-\cup\{(P:M)|(P:M) \in\{(P_i:M)\}_{i \in I} \cup \{(Q_j:M)\}_{j \in J} \}
      \end{align*}
      Now, let $P \in \widetilde{D}(S \cap T)$. Then,
       there exist $i_0 \in I$ or $j_0 \in J$ such that $(P:M) \subseteq (P_{i_0}:M) \subseteq \bigcup_{i \in I}(P_i:M)$ or $(P:M) \subseteq (Q_{j_0}:M) \subseteq \bigcup_{j \in J}(Q_j:M)$ since $\{(P_i:M)\}_{i \in I}\cup \{(Q_j:M)\}_{j \in J}$ satisfies $PA$. In that case, $(P:M) \cap [R-\bigcup_{i \in I}(P_i:M)]= \emptyset$ or $(P:M) \cap [R-\bigcup_{j \in J}(Q_j:M)]= \emptyset$. Hence, $P \in \widetilde{D}(S) \cup \widetilde{D}(T)$. Therefore, $\widetilde{D}(S \cap T) \subseteq \widetilde{D}(S) \cup \widetilde{D}(T)$. On the other hand, for any  $P \in  \widetilde{D}(S) \cup \widetilde{D}(T)$ it can be easily seen that $P \in \widetilde{D}(S \cap T)$ since $ S \cap T=R-\cup\{(P:M)|(P:M) \in\{(P_i:M)\}_{i \in I} \cup \{(Q_j:M)\}_{j \in J} \}$. Hence, $\widetilde{D}(S) \cup \widetilde{D}(T) \subseteq \widetilde{D}(S \cap T)$. As a result, $\widetilde{D}(S) \cup \widetilde{D}(T)=\widetilde{D}(S \cap T)$.
       \item[(ii)] Suppose that $\{(P_i:M)\}_{i \in I}\subseteq \operatorname{Min}(R)$. Let $K=\{P \in \operatorname{Spec}(M)|\ (P:M)=(P_{i_0}:M) \ \text{for some} \ i_0 \in I\}$ and $Q \in \widetilde{D}(S)$. Then, $(Q:M) \cap [R- \bigcup_{i\in I} (P_i:M)]=\emptyset$ so $(Q:M) \subseteq (P_{i_0}:M)$ for some $i_0 \in I$ since  $\{(P_i:M)\}_{i \in i}$ satisfies $PA$. Hence, $(Q:M)=(P_{i_0}:M)$ by our assumption. Therefore, $\widetilde{D}(S) \subseteq K$. Now, we assume that $\widetilde{D}(S)=\{P \in \operatorname{Spec}(M)|(P:M)=(P_{i_0}:M) \ \text{for some} \ i_0 \in  I\}$. Now, we will show that $(P_{i}:M)$ is a minimal prime ideal for each $i\in I$. Let $Q$ be a minimal prime ideal such that $Q\subseteq (P_i:M)$. Since $M$ is a finitely generated faithful module, by \cite[Proposition 8]{Lu2}, we have $(QM:M)=Q$. On the other hand, as $M$ is finitely generated, by \cite[Theorem 3.3]{McMo}, there exists a prime submodule $P^{*}$ such that $(P^{*}:M)=Q$. This gives $P^{*}\in \widetilde{D}(S)$. Since $(P_i:M)$'s are incomparable primes, we conclude that $Q=(P^{*}:M)=(P_i:M)$ which implies that $(P_i:M)$ is a minimal prime ideal of $R$.
       \item[(iii)] Suppose that $\operatorname{Max}(R)=\{(P_i:M)\}_{i \in I}$ and choose $P \in \widetilde{\operatorname{Spec}(M)}$. Then, there exists a maximal ideal $P_{i_0}$ such that $(P:M) \subseteq (P_{i_0}:M)$ for some $i_0 \in I$. This implies that $(P:M) \cap [R-\bigcup_{i \in I}(P_i:M)]= \emptyset$, i.e., $P \in \widetilde{D}(S)$ for some multiplicative subset $S$ of $R$. Hence, $\widetilde{D}(S)=\widetilde{\operatorname{Spec}(M)}$. 
  \end{enumerate}  
\end{proof}
For any set $X$, recall from \cite{ParsaTopb} that a binary relation $\mathcal{R}$ on $X$ is said to be a connected relation if for any $x,y \in X$, there exist elements $x_1,x_2,...,x_n$ of $X$ such that $x=x_1,y=x_n$ and $x_{i-1} \mathcal{R}x_i$ or $x_i\mathcal{R}x_{i-1}$ for all $2 \leq i \leq n$. In this case, the sequence $x_1,x_2,...,x_n$ is said to be a path from $x$ to $y$. 
\begin{theorem}\label{connected}
Let $M$ be an $R$-module. If inclusion $\subseteq$ is a connected relation on the set of prime submodules of $M$, then $\widetilde{\operatorname{Spec}(M)}$ is a connected space. In particular, the converse holds if $M$ is a multiplication module.
\end{theorem}
\begin{proof}
Assume that $\subseteq$ is a connected relation on $\operatorname{Spec}(M)$. Now, we will show that $\widetilde{\operatorname{Spec}(M)}$ is a connected space. Let $U$ be a nonempty clopen set $U$ in $\widetilde{\operatorname{Spec}(M)}$. It is sufficient to show that $U=\widetilde{\operatorname{Spec}(M)}$. Choose $Q\in \widetilde{\operatorname{Spec}(M)}-U$ and $P\in U$. Since $\subseteq$ is a connected relation, there exists a path consisting of prime submodules $P_1=P, P_2,\ldots, P_n=Q$, where $P_i\subseteq P_{i-1}$ or $P_{i-1}\subseteq P_i$ for each $i=1,2,\ldots,n$. Choose the greatest integer $t$ such that $P_t\in U$. Then note that $P_{t+1}\in\widetilde{\operatorname{Spec}(M)}-U$. Now, we have two cases. \textbf{Case 1:} Let $P_{t+1}\subseteq P_t$. Since $U$ is open and $P_t\in U$, we have $P_{t+1}\in\widetilde{D}(R-(P_t:M))\subseteq U$ which is a contradiction. \textbf{Case 2:} Let $P_t\subseteq P_{t+1}$. Since $\widetilde{\operatorname{Spec}(M)}-U$ is open and $P_{t+1}\in \widetilde{\operatorname{Spec}(M)}-U$, we conclude that $P_t\in\widetilde{D}(R-(P_{t+1}:M))\subseteq\widetilde{\operatorname{Spec}(M)}-U$ which is again a contradiction. Thus, $U=\widetilde{\operatorname{Spec}(M)}$ and the clopen sets in $\widetilde{\operatorname{Spec}(M)}$ are the only $\emptyset$ and $\widetilde{\operatorname{Spec}(M)}$. Hence, $\widetilde{\operatorname{Spec}(M)}$ is a connected space. For the converse, assume that $\widetilde{\operatorname{Spec}(M)}$ is a connected space and $M$ is a multiplication module. Now, we will show that $\subseteq$ is a connected relation. Suppose to the contrary that $\subseteq$ is not a connected relation. Then there exist prime submodules $P,Q$ of $M$ such that there exist no path between $P$ and $Q$. Let $U_P$ be the set of prime submodules $N$ of $M$ such that there exists a path between $N$ and $P$. First, we will show that $U_P$ is a open set in $\widetilde{\operatorname{Spec}(M)}$. To see this, take a point $K\in U_P$. Then there exists a path between $K$ and $P$. Let $N\in \widetilde{D}(R-(K:M))$. Then we have $(N:M)\subseteq (K:M)$. Since $M$ is a multiplication module, we have $N\subseteq K$, and so there exists a path between $N$ and $P$. This implies that $N\in U$. Thus, $U_P$ is an open set in $\widetilde{\operatorname{Spec}(M)}$. Now, we will show that $\widetilde{\operatorname{Spec}(M)}-U_P$ is also open. Let $K\in\widetilde{\operatorname{Spec}(M)}-U_P$. If $\widetilde{D}(R-(K:M))\nsubseteq\widetilde{\operatorname{Spec}(M)}-U_P$, then there exists $N\in\widetilde{D}(R-(K:M))$ such that $N\in U_P$. Since $M$ is a multiplication module and $N\in \widetilde{D}(R-(K:M))$, we have $N\subseteq K$ and there exists a path between $N$ and $P$. Thus, there exists a path between $K$ and $P$, that is $K\in U_P$ which is a contradiction. Then we have $\widetilde{D}(R-(K:M))\subseteq\widetilde{\operatorname{Spec}(M)}-U_P$, that is, $\widetilde{\operatorname{Spec}(M)}-U_P$ is an open set. This contradicts with the connectedness of $\widetilde{\operatorname{Spec}(M)}$. Hence, $\subseteq$ is a connected relation. 
\end{proof}
Recall that an $R$-module $M$ is said to be uniserial if its submodules are linearly ordered by inclusion \cite{Salce}.
\begin{corollary}\label{uniconnec} Let $M$ be an $R$-module. The following statements are satisfied.
\begin{enumerate}
    \item[(i)] If $M$ is a uniserial module, then $\widetilde{\operatorname{Spec}(M)}$ is a connected space. 
    \item[(ii)] If $M$ is a vector space, then $\widetilde{\operatorname{Spec}(M)}$ is a connected space. 
    \item[(iii)] If $M$ is a zero dimensional module with $|\operatorname{Spec}(M)|\geq 2$, then $\widetilde{\operatorname{Spec}(M)}$ is not a connected space. 
\end{enumerate}
\end{corollary}
\begin{proof}
 $(i),(ii):$ Note that $\subseteq$ is a connected relation, and the claim follows from Theorem \ref{connected}.\\
 $(iii):$ Note that $\subseteq$ is not a connected relation, and the claim follows from Theorem \ref{connected}.
\end{proof}
\begin{theorem}\label{intcl}
Let $M$ be an $R$-module and $N \subseteq \widetilde{\operatorname{Spec}(M)}$. Then
\begin{enumerate}
    \item[(i)] $\overline{N} = \{Q \in \operatorname{Spec}(M)|\ (P:M) \subseteq (Q:M)\ \text{for some}\ P\in N \}$. 
    \item[(ii)] $\overline{\bigcup_{i \in I}N_i}=\bigcup_{i \in I}\overline{N_i}$ for every family $\{N_i\}_{i \in I} \subseteq \widetilde{\operatorname{Spec}(M)}$.
    \item[(iii)]  $N^\circ=\{P \in N|\ \widetilde{D}(R-(P:M)) \subseteq N\}$.
    \item[(iv)] Let $M$ be a multiplication module. Then  $$N^{'}=\overline{N}-\{P \in N\ |\ P\ \text{is a  minimal element of}\ N\}.$$
\end{enumerate}
\end{theorem}
\begin{proof}
 \begin{enumerate}
     \item[(i)] Choose $Q \in \overline{N}$ and then one may have $\widetilde{D}(R-(Q:M)) \cap N \neq \emptyset$. It follows that there exists a prime submodule $P$ of $M$ such that $P \in \widetilde{D}(R-(Q:M)) \cap N$ implying that $(P:M) \subseteq (Q:M)$. Thus, $\overline{N} \subseteq \{Q \in \operatorname{Spec}(M)| \ (P:M) \subseteq (Q:M)\ \text{for some}\ P\in N \}$. Now, let $T \in \{Q \in \operatorname{Spec}(M)| \ (P:M) \subseteq (Q:M)\ \text{for some}\ P\in N \}$. Then there exists a prime submodule $P$ of $M$ such that $(P:M) \subseteq (T:M)$ which implies that $P \in \widetilde{D}(R-(T:M))$. Moreover, one can conclude that $\widetilde{D}(R-(T:M)) \cap N \neq \emptyset$, that is $T \in \ \overline{N}$.
     \item[(ii)]  The containment $\bigcup_{i \in I}\overline{N_i}\subseteq\overline{\bigcup_{i \in I}N_i}$ is always true. For the other containment, let $Q \in \overline{\bigcup_{i \in I}N_i}$. Then there exists a prime submodule $P\in \bigcup_{i \in I}N_i$ such that $(P:M) \subseteq (Q:M)$. Then there exists $i_0\in I$ such that $P \in N_{i_0}$ and $(P:M) \subseteq (Q:M)$. This implies that $Q \in \overline{N_{i_0}}\subseteq\bigcup_{i \in I}\overline{N_i}$, which completes the proof.
     \item[(iii)] Let $Q \in N^\circ$. Then, there exists an open set $U \subseteq N$ such that $Q \in U$ which implies that $D(R-(Q:M)) \subseteq U$ by Lemma \ref{neigh}. Thus, $D(R-(Q:M)) \subseteq N$, that is $Q \in \{P \in N|\ \widetilde{D}(R-(P:M)) \subseteq N\}$. Hence, $N^\circ\subseteq\{P \in N|\ \widetilde{D}(R-(P:M)) \subseteq N\}$. The other containment is true since $N^\circ$ is the largest open set contained in $N$.
     \item[(iv)] Let $P \in N^{'}$.  Then we have $[\widetilde{D}(R-(P:M))-\{P\}] \cap N \neq \emptyset$. This implies that there exists $K\in N-\{P\}$ such that $K\in\widetilde{D}(R-(P:M))$. Moreover, we get $(K:M)\subseteq (P:M)$. Since $M$ is a multiplication module and $N\neq P$, we have $K\subsetneq P$. Hence, $N^{'} \subseteq \overline{N}-\{P \in N|\ P \ \text{is a minimal element of} \ N\}$. Let $Q\in\overline{N}-\{P \in N|\ P \ \text{is a minimal element of} \ N\}$. It means that $\widetilde{D}(R-(Q:M)) \cap N \neq \emptyset$. Moreover, if $\widetilde{D}(R-(Q:M)) \cap N = \{Q\}$, then $Q$ is a minimal element in $N$ which is a contradiction. Hence $[\widetilde{D}(R-(Q:M))] \cap (N-\{Q\}) \neq \emptyset$, that is $Q \in N^{'}$. As a result, $\overline{N}-\{P \in N|\ P \ \text{is a minimal element of} \ N\} \subseteq N^{'}$.
 \end{enumerate}   
\end{proof}
\section{Separation Axioms in $\widetilde{\operatorname{Spec}(M)}$}
One of the fundamental concepts in a topology is classifying the separation axioms for a given topology. In this section, we investigate each separation axiom for $\widetilde{\operatorname{Spec}(M)}$. Recall from \cite{munkres} that a topological space $(X,\tau)$ is a \textit{$T_0$-space} if for every distinct points, there is an open set containing precisely one of them. Moreover, $X$ is called a \textit{$T_1$-space} if for any distinct points $x,y\in X$, there is an open set $O$ such that $x\in O$ and $y\not\in O$. Equivalently, $\overline{\{x\}}=\{x\}$ for every $x\in X$.
\begin{theorem}\label{T0}
$\widetilde{\operatorname{Spec}(M)}$ is $T_0$-space if and only if $(P:M)\neq (Q:M)$ for all distinct prime submodules $P,Q\in \operatorname{Spec}(M)$.  
\end{theorem}
\begin{proof}
Suppose that $\widetilde{\operatorname{Spec}(M)}$ is a $T_0$-space and take two distinct points $P,Q \in \operatorname{Spec}(M)$. Since $\widetilde{\operatorname{Spec}(M)}$ is a $T_0$-space, we have $\overline{\{P}\}\neq\overline{\{Q}\}$. This implies that $(P:M)\neq (Q:M)$ by Theorem \ref{intcl}. For the converse, assume that $(P:M)\neq (Q:M)$ for all distinct prime submodules $P,Q\in \operatorname{Spec}(M)$. Now, we will show that $\overline{\{P}\}\neq\overline{\{Q}\}$. Suppose that $\overline{\{P}\}=\overline{\{Q}\}$. Since $P\in \overline{\{P}\}=\overline{\{Q}\}$, by Theorem \ref{intcl}, we conclude that $(Q:M)\subseteq (P:M)$. Similarly, we get $(P:M)\subseteq (Q:M)$, that is $(P:M)=(Q:M)$ which is a contradiction. Hence, $\widetilde{\operatorname{Spec}(M)}$ is a $T_0$-space.
\end{proof}
\begin{example}
Consider the $\mathbb Z$-module $\mathbb Z^2$. One can see that $\widetilde{Spec}(\mathbb Z)$ is not a $T_0$-space since distinct prime submodules $P=0\times 0$ and $Q=0\times \mathbb Z$ have the same colon ideal $(P:\mathbb Z)=(Q:\mathbb Z)=0$. 
\end{example}
Recall from \cite{Azizimult} that an $R$-module $M$ is said to be a weak multiplication if $P=(P:M)M$ for every prime submodule $P$ of $M$.

\begin{corollary}
Let $M$ be a weak multiplication module. Then $\widetilde{\operatorname{Spec}(M)}$ is a $T_0$-space. In particular, if $M$ is a multiplication module, then $\widetilde{\operatorname{Spec}(M)}$ is a $T_0$-space. 
\end{corollary}
\begin{proof}
If $M$ is a weak multiplication module, then for every distinct prime submodule $P,Q$ of $M$, we have $(P:M)\neq (Q:M)$. The rest follows from Theorem \ref{T0}.
\end{proof}

\begin{theorem}\label{T1}
$\widetilde{\operatorname{Spec}(M)}$ is a $T_1$-space if and only if $\{(Q:M)| \ Q\in \operatorname{Spec}(M)\}$ is the set of incomparable prime ideals.
\end{theorem}
\begin{proof}
We first note that $\overline{\{P\}}=\{Q\in \operatorname{Spec}(M)\ |\ (P:M)\subseteq (Q:M)\}$ by Theorem \ref{intcl} (i). Now, suppose that $\widetilde{\operatorname{Spec}(M)}$ is a $T_1$-space. Now, we will show that $\{(Q:M)|\ Q\in \operatorname{Spec}(M)\}$ is the set of incomparable prime ideals. Take two distinct prime submodules $P,Q$ of $M$. Now, we will show that $(P:M)\nsubseteq (Q:M)$ and $(Q:M)\nsubseteq (P:M)$. Without loss of generality, assume that $(P:M)\subseteq (Q:M)$. Then by Theorem \ref{intcl} (i), we have $Q\in \overline{\{P\}}$. Since $\widetilde{\operatorname{Spec}(M)}$ is a $T_1$-space, we have $Q\in \overline{\{P\}}=\{P\}$ which is a contradiction. Hence, $\{(Q:M)|\ Q\in \operatorname{Spec}(M)\}$ is the set of incomparable prime ideals. For the converse, assume that $\{(Q:M) |\ Q\in \operatorname{Spec}(M)\}$ is the set of incomparable prime ideals. Now, we will show that $\overline{\{P\}}=\{P\}$ for every $P\in \operatorname{Spec}(M)$. Let $P\in \operatorname{Spec}(M)$ and $Q\in \overline{\{P\}}$. Then by Theorem \ref{intcl} (i), we have $(P:M)\subseteq (Q:M)$. Then by assumption, we conclude that $P=Q$, that is, $\overline{\{P\}}=\{P\}$. Hence, $\widetilde{\operatorname{Spec}(M)}$ is a $T_1$-space.
\end{proof}
We say a topological space $(X,\tau)$ is called a \textit{$T_2$-space} or \textit{Hausdorff} if for every distinct point $x,y\in X$, there are open sets $O$ and $O'$ such that $x\in O$, $y\in O'$ and $O\cap O'=\emptyset$ \cite{Ricmond}. The following corollary can be obtained immediately by Theorem \ref{T1} together with the fact that Alexandrov $T_1$-spaces are discrete space.
\begin{theorem}\label{Discrete}
Let $M$ be an $R$-module. Then the following statements are equivalent:
\begin{enumerate}
    \item[(i)] $\widetilde{\operatorname{Spec}(M)}$ is discrete.
    \item[(ii)] $\{(Q:M)|\ Q\in \operatorname{Spec}(M)\}$ is the set of incomparable prime ideals.
    \item[(iii)] $\widetilde{\operatorname{Spec}(M)}$ is a $T_1$-space.
    \item[(iv)] $\widetilde{\operatorname{Spec}(M)}$ is Hausdorff.
\end{enumerate}
\end{theorem}
\begin{proof}
    $(i) \Rightarrow (ii)\Rightarrow (iii)$: Follows from Theorem \ref{T1}. \\
    $(iii) \Rightarrow (i)\Rightarrow (iv):$ Follows from the fact that Alexandrov $T_1$-spaces are discrete.\\
    $(iv) \Rightarrow (iii)$: Clear.
\end{proof}
A topological space $(X,\tau)$ is called a \textit{$T_3$-space} if for any closed set $C$ and $x\in X-C$, there are disjoint open sets $U$ and $V$ so that $C\subseteq O$ and $x\in V$ \cite{munkres}.
\begin{theorem}\label{T3}
 Let $M$ be an $R$-module. $\widetilde{\operatorname{Spec}(M)}$ is a  $T_3$-space if and only if $\widetilde{D}(R-(P:M))$ is closed for all $P \in \widetilde{\operatorname{Spec}(M)}$.   
\end{theorem}
\begin{proof}
 $(\Rightarrow)$: Suppose that $\widetilde{\operatorname{Spec}(M)}$ is a $T_3$-space and choose $P \in \widetilde{\operatorname{Spec}(M)}$. Then there exists an open subset $V$ of $\widetilde{\operatorname{Spec}(M)}$ containing $\widetilde{\operatorname{Spec}(M)}-\widetilde{D}(R-(P:M))$ such that $V \cap \widetilde{D}(R-(P:M))= \emptyset$. This gives $V\subseteq\widetilde{\operatorname{Spec}(M)}-\widetilde{D}(R-(P:M))$ which implies that $V=\widetilde{\operatorname{Spec}(M)}-\widetilde{D}(R-(P:M))$ is an open set, that is $\widetilde{D}(R-(P:M))$ is closed.\\
 $(\Leftarrow):$ It is sufficient to show that every closed set is open, or equivalently every open set is closed. Let $O$ be an open set in $\widetilde{\operatorname{Spec}(M)}$. Then by Proposition \ref{basis2}, we have $O=\bigcup_{i \in I}\widetilde{D}(R-(P_i:M))$. On the other hand, by the assumption, $\widetilde{D}(R-(P_i:M))$ is closed. Then by Theorem \ref{intcl} (ii), we have $\overline{O}=\overline{\bigcup\widetilde{D}(R-(P_i:M)}=\bigcup{\overline{\widetilde{D}(R-(P_i:M)}}=\bigcup{\widetilde{D}(R-(P_i:M)}=O$ which implies that $O$ is closed in $\widetilde{\operatorname{Spec}(M)}$. Hence, $\widetilde{\operatorname{Spec}(M)}$ is a $T_3$-space.
\end{proof}

\section{Construction of a Lindelöf Space which is not quasi-compact}

This section examines the quasi-compact, locally compact, and Lindelöf properties of $\widetilde{\operatorname{Spec}(M)}$. The final part presents a method for constructing a Lindelöf space that is not quasi-compact. A topological space $(X,\tau)$ is called a Lindelöf space if every open cover of $X$ has a countable subcover. Moreover, $X$ is called quasi-compact if every open cover of $X$ has a finite subcover. Additionally, $X$ is locally compact if every point $x \in X$ has a compact neighborhood \cite{munkres}. Since $\widetilde{\operatorname{Spec(M)}}$ is an Alexandrov space, one can see that $\widetilde{\operatorname{Spec}(M)}$ is a locally compact space. It is well known that the Zariski topology on $\operatorname{Spec}(R)$, the spectrum of a ring $R$, is always quasi-compact, which implies that $\operatorname{Spec}(R)$ is locally compact and Lindelöf. However, $\widetilde{\operatorname{Spec}(R)}$ is not necessarily quasi-compact or Lindelöf, as demonstrated by the following example.
\begin{example}
Let $R = C(\mathbb{R})$ denote the ring of all real-valued continuous functions. Each maximal ideal of $R$ is of the form $M_x = \{f \in C(\mathbb{R}) : f(x) = 0\}$ for some $x \in \mathbb{R}$. Therefore, $R$ possesses uncountably many maximal ideals. Define $S_x = R \setminus M_x$. It follows that $\widetilde{\operatorname{Spec}(R)} = \bigcup_{x \in \mathbb{R}} \widetilde{D}(S_x)$. Consequently, $\widetilde{\operatorname{Spec}(R)}$ is not a Lindelöf space.
\end{example}

\begin{theorem}
Let $M$ be an $R$-module, $\{P_i\}_{i \in I}$ be a collection of prime submodules of $M$ and $S=R-\bigcup_{i \in I}(P_i:M)$. If $I$ is finite, then $\widetilde{D}(S)$ is a quasi-compact subspace of $\widetilde{\operatorname{Spec}(M)}$.
\end{theorem}

\begin{proof}
Note that 
       \begin{align*}
           \widetilde{D}(S)=\{ P \in \operatorname{Spec}(M)|\ (P:M) \cap [R-\bigcup_{i \in I}(P_i:M)]=\emptyset\}\\
           = \{ P \in \operatorname{Spec}(M)|\ (P:M) \subseteq \bigcup_{i \in I}(P_i:M)\} \\
           = \{ P \in \operatorname{Spec}(M)|\ (P:M) \subseteq (P_{i_0}:M), \text{ for some } i_0 \in I\}
       \end{align*}
       since $\{(P_i:M)\}_{i \in I}$ satisfies $PA$. Let $\mathcal{A} \subseteq \widetilde{\operatorname{Spec}(M)}$ and $\widetilde{D}(S) \subseteq \bigcup_{N \in \mathcal{A}}\widetilde{D}(R-(N:M))$. Then for any $P \in \widetilde{D}(S)$, we have $(P:M) \cap [R-(N:M)]= \emptyset$ for some $N \in \mathcal{A}$. Since for any $i_0 \in I$, there exists $N_{i_0} \in \mathcal{A}$ such that $(P_{i_0}:M)\subseteq (N_{i_0}:M)$, we get $\widetilde{D}(S) \subseteq \bigcup_{{i_0} \in I}\widetilde{D}(R-(N_{i_0}:M))$. Hence, $\widetilde{D}(S)$ is a quasi-compact subspace of $\widetilde{\operatorname{Spec}(M)}$. 
\end{proof}

\begin{corollary}\label{comp}
$\widetilde{\operatorname{Spec}(M)}$ is a compact space if and only if $\widetilde{\operatorname{Spec}(M)}= \\ \bigcup_{i=1}^n \widetilde{D}(R-(P_i:M))$ for some collections of prime submodules $\{P_i\}_{i=1}^{n}$. 
\end{corollary}
\begin{proof}
The first direction of the statement is clear. Conversely, assume that $\widetilde{\operatorname{Spec}(M)}=\bigcup_{i=1}^n \widetilde{D}(R-(P_i:M))$ for some collections of prime submodules $\{P_i\}_{i=1}^{n}$. Now, put $S=R-\bigcup_{i=1}^{n}(P_i:M)$. Then note that $\widetilde{D}(S)=\widetilde{\operatorname{Spec}(M)}$. The rest follows from previous theorem. 
\end{proof}
In \cite[Corollary 11]{ParsaTopb} the authors showed that if $R$ is a semi-local ring, then $\widetilde{\operatorname{Spec}(R)}$ is a compact space. The following theorem improves the result in \cite[Corollary 11]{ParsaTopb}.
\begin{theorem}\label{tcomp}
Let $M$ be a finitely generated faithful module over a ring $R$. Then $\widetilde{\operatorname{Spec}(M)}$ is a compact space if and only if $R$ is a semi-local ring.
\end{theorem}
\begin{proof}
Suppose that $\widetilde{\operatorname{Spec}(M)}$ is a compact space. Then by Corollary \ref{comp}, $\widetilde{\operatorname{Spec}(M)}=\bigcup_{i=1}^n \widetilde{D}(R-(P_i:M))$ for some collections of prime submodules $\{P_i\}_{i=1}^{n}$. Choose a maximal ideal $Q$ of $R$. Since $M$ is a finitely generated faithful module, by \cite[Proposition 8]{Lu2}, $(QM:M)=Q$ is a maximal ideal. Thus $QM$ is a prime submodule of $M$. This gives $QM\in\widetilde{\operatorname{Spec}(M)}=\bigcup_{i=1}^n \widetilde{D}(R-(P_i:M))$. Then there exists $i=1,2,\ldots,n$ such that $(QM:M)=Q\subseteq (P_i:M)$. Again since $Q$ is a maximal ideal, we have $Q=(P_i:M)$. Hence, $R$ is a semi-local ring. For the converse, assume that $R$ is a semi-local ring, say $\operatorname{Max}(R)=\{Q_1,Q_2,\ldots,Q_k\}$. Since $M$ is a finitely generated faithful module, by \cite[Proposition 8]{Lu2}, $Q_1M,Q_2M,\ldots,Q_kM$ are prime submodules of $M$ and $(Q_iM:M)=Q_i$ for each $i=1,2,\ldots,k$. Also, it is easy to see that $\widetilde{\operatorname{Spec}(M)}=\bigcup_{i=1}^n \widetilde{D}(R-Q_i)=\bigcup_{i=1}^n \widetilde{D}(R-(Q_iM:M))$. Hence by Corollary \ref{comp}, $\widetilde{\operatorname{Spec}(M)}$ is a compact space. 
\end{proof}

\begin{theorem}\label{lindelöf}
$\widetilde{\operatorname{Spec}(M)}$ is a Lindelöf space if and only if $\widetilde{\operatorname{Spec}(M)}=\bigcup_{i=1}^\infty \widetilde{D}(R-(P_i:M))$ for some collections of prime submodules $\{P_i\}_{i=1}^{\infty}$. 
\end{theorem}

\begin{proof}
 The first direction of the statement follows from Proposition \ref{basis2} ii). Now, suppose $\widetilde{\operatorname{Spec}(M)}=\bigcup_{i=1}^\infty \widetilde{D}(R-(P_i:M))$ for some collections of prime submodules $\{P_i\}_{i=1}^{\infty}$. Assume moreover that $\operatorname{Spec}(M)=\bigcup_{j\in J}O_j$ for some open sets $O_j$ in $\widetilde{\operatorname{Spec}(M)}$. Then, there exist open sets $O_i$ such that $P_i\in O_i$ for each $P_i\in \operatorname{Spec}(M)$. It follows that $\widetilde{D}(R-(P_i:M))\subseteq O_i$ by Lemma \ref{neigh}. Therefore, $\operatorname{Spec}(M)=\cup_{i=1}^\infty O_i$.
\end{proof}
\begin{theorem}\label{tLin}
Let $M$ be a finitely generated faithful module over a ring $R$. Then $\widetilde{\operatorname{Spec}(M)}$ is a Lindelöf space if and only if $R$ has countably many maximal ideals. 
\end{theorem}
\begin{proof}
The result follows from the previous theorem and employing an argument analogous to that used in the proof of Theorem \ref{tcomp}.
\end{proof}
Recall that a free module $M$ over a ring $R$, every basis of $M$ has the same cardinality, and this is denoted by $\operatorname{rank}_{R}(M)$. If $\operatorname{rank}_{R}(M)<\infty$, $M$ is said to be a free module of finite rank \cite{Sharp}. Now, we are ready to give a method for constructing a Lindelöf space which is not compact. 
\begin{corollary}
\begin{enumerate}
    \item[(i)] Let $M$ be a finitely generated faithful module over a ring $R$ with countably infinitely many maximal ideals. Then $\widetilde{\operatorname{Spec}(M)}$ is a Lindelöf space which is not a quasi-compact.
    \item[(ii)] Let $M$ be a free module of finite rank over the ring $R$ with countably infinitely many maximal ideals. Then $\widetilde{\operatorname{Spec}(M)}$ is a Lindelöf space which is not quasi-compact.
\end{enumerate}
\end{corollary}
\begin{proof}
Follows from Theorem \ref{tcomp} and Theorem \ref{tLin}.
\end{proof}

\begin{example}
\begin{enumerate}
    \item[(i)] Let $n>1$ be a fixed integer, and consider $\mathbb{Z}$-module $M=\mathbb{Z}^n$. Then $\widetilde{\operatorname{Spec}(M)}$ is a Lindelöf space which is not quasi-compact.
    \item[(ii)] Let $K$ be a countably infinite field and $R=K[X]$, where $X$ is an indeterminate over $K$, and consider $M=R^n$ for some integer $n>1$. Then $\widetilde{\operatorname{Spec}(M)}$ is a Lindelöf space which is not quasi-compact. In particular, 
    the $\widetilde{\operatorname{Spec}(M)}$ topology of $R=\mathbb{Z}_p[X]$-module $R^n$ is Lindelöf but not a quasi-compact space for every prime number $p$ and every integer $n>1$.
\end{enumerate}
\end{example}

\end{document}